\documentclass{amsart}
\usepackage{amsfonts,amssymb,amscd,amsmath,enumerate,verbatim,calc}
\usepackage[all]{xy}

\newcommand{\wrt}{with respect to}

\newcommand{\G}{\mathcal{G} }
\newcommand{\Ac}{\mathcal{A}}

\newcommand{\m}{\mathfrak{m} }
\newcommand{\M}{\mathcal{M} }

\newcommand{\C}{\mathcal{C} }
\newcommand{\V}{\mathcal{V} }
\newcommand{\K}{\mathcal{K} }
\newcommand{\Hc}{\mathcal{H} }
\newcommand{\D}{\mathcal{D} }

\newcommand{\Z}{\mathbb{Z} }

\newcommand{\rt}{\rightarrow}

\newcommand{\Xb}{\mathbf{X}_\bullet}
\newcommand{\Zb}{\mathbf{Z}_\bullet}
\newcommand{\Yb}{\mathbf{Y}_\bullet}
\newcommand{\Kb}{\mathbf{K}_\bullet}

\newcommand{\Pb}{\mathbf{P}_\bullet}
\newcommand{\Qb}{\mathbf{Q}_\bullet}
\newcommand{\Fb}{\mathbf{F}_\bullet}
\newcommand{\Wb}{\mathbf{W}_\bullet}
\newcommand{\Ub}{\mathbf{U}_\bullet}
\newcommand{\Vb}{\mathbf{V}_\bullet}

\newcommand{\image}{\operatorname{image}}

\newcommand{\Supp}{\operatorname{Supp}}

\newcommand{\Spec}{\operatorname{Spec}}

\newcommand{\mmod}{\operatorname{mod}}
\newcommand{\Hom}{\operatorname{Hom}}

\theoremstyle{plain}

\newtheorem{theorem}{Theorem}[section]
\newtheorem{corollary}[theorem]{Corollary}
\newtheorem{lemma}[theorem]{Lemma}
\newtheorem{proposition}[theorem]{Proposition}

\theoremstyle{definition}

\newtheorem{s}[theorem]{}

\theoremstyle{remark}

\begin{document}

\title[Dimension filtration]{Dimension filtration of the bounded Derived category of a Noetherian ring}
\author{Tony~J.~Puthenpurakal}
\date{\today}
\address{Department of Mathematics, IIT Bombay, Powai, Mumbai 400 076}

\email{tputhen@math.iitb.ac.in}
\subjclass{Primary 13D09  ; Secondary 13D02}
\keywords{bounded derived category, bounded t-structures, Krull-Remak-Schmidt categories }

 \begin{abstract}
Let $A$ be a Noetherian ring of dimension $d$ and let  $\D^b(A)$ be the bounded derived category of $A$. Let $\D_i^b(A)$ denote the thick subcategory of $\D^b(A)$ consisting of complexes $\Xb$ with $\dim H^n(\Xb) \leq i$ for all $n$. Set $\D_{-1}^b(A) = 0$. Consider the Verdier quotients $\C_i(A) = \D_i^b(A)/\D_{i-1}^b(A)$. We show for $i = 0, \ldots, d$, $\C_i(A)$ is a Krull-Remak-Schmidt triangulated category with a bounded $t$-structure. We identify its heart. We also prove that if $A$ is regular then $\C_i(A)$ has AR-triangles.
We also prove that
$$ \C_i(A) \cong \bigoplus_{\stackrel{P}{\dim A/P = i}} \D_0^b(A_P). $$
\end{abstract}
 \maketitle
\section{introduction}
Let $A$ be a Noetherian ring of dimension $d$ and let  $\D^b(A)$ be the bounded derived category of $A$.  Let $\D_i^b(A)$ denote the thick subcategory of $\D^b(A)$ consisting of complexes $\Xb$ with $\dim H^n(\Xb) \leq i$ for all $n$. Set $\D_{-1}^b(A) = 0$.
We have a filtration of thick subcategories of $\D^b(A)$
\[
0 \subseteq \D^b_0(A) \subseteq  \D^b_1(A) \subseteq \cdots \subseteq \D^b_i(A) \subseteq \cdots \subseteq \D_{d-1}^b(A) \subseteq \D_d^b(A) = \D^b(A).
\]
It is natural to investigate the Verdier quotients $\C_i(A) = \D^b_i(A)/\D^b_{i-1}(A)$ for $i = 0, \ldots, d$.
We first prove
\begin{theorem}\label{t-struc}
The triangulated categories $\C_i(A)$ has a bounded $t$-structure for $i = 0, 1, \ldots, d$.
\end{theorem}
The heart of a triangulated category with a bounded $T$-structure is abelian. It is natural to investigate what is the heart of $\C_i(A)$. To understand we need a few preliminaries.
 Let $\mmod(A)$ denote the category of all finitely generated $A$-modules. Let $\M_i(A)$ denote the category of all finitely generated $A$-modules of dimension $\leq i$. We note that $\M_i(A)$ is a Serre-subcategory of $\mmod(A)$. We prove
 \begin{theorem}\label{heart}
 The heart of $\C_i(A)$  (\wrt \ the $t$-structure in \ref{t-struc}) is isomorphic to $\M_i(A)/\M_{i-1}(A)$.
 \end{theorem}
Next we prove the following surprising result.
\begin{theorem}\label{krs}
$\C_i(A)$ is Krull-Remak-Schmidt for $i = 0, \ldots, d$.
\end{theorem}
As an application of Theorem \ref{krs}
we show
\begin{theorem}\label{decomp}For $i = 0, \ldots, d$;
we have an equivalence of categories
$$ \C_i(A) \cong \bigoplus_{\stackrel{P}{\dim A/P = i}} \D_0^b(A_P). $$
\end{theorem}
As an application of this result we show:
\begin{corollary}\label{ar}
If $A$ is regular then $\C_i(A)$ has Auslander-Rieten triangles for all $i$.
\end{corollary}

Here is an overview of the contents of this paper. In section two we prove Theorem \ref{t-struc}. In section three we prove Theorem \ref{heart} In the next section we prove Theorem \ref{krs}.
Finally in section five we prove Theorem \ref{decomp} and Corollary \ref{ar}.
\section{$t$-structure}
Let $\G$ be a triangulated category with its shift functor denoted by $[1]$.

\s Recall
that a $t$-structure on $\G$ is a pair of strictly (i.e. closed under isomorphisms) full additive subcategories
$(\G^{\leq 0}, \G^{\geq 0})$ satisfying the following conditions:\\
(T1) $\Hom_\G(X, Y[-1]) = 0 $ for all $X \in \G^{\leq 0}$ and $Y \in \G^{\geq 0}$.\\
(T2) $\G^{\leq 0}$ is closed under the functor $[1]$, and
 $\G^{\geq 0}$ is closed under the functor $[-1]$. \\
(T3) for each $X \in \G$, there is an exact triangle
$ A \rt X \rt B[-1]\rt A[1]$ with $A \in \G^{\leq 0}$ and
$B \in \G^{\geq 0}$.

For $n \in \Z$,  set $\G^{\leq n} = \G^{\leq 0}[-n] $ and $\G^{\geq n} = \G^{\geq 0}[-n] $. A $t$-structure  if for each $X \in \G$ there exists $m, n $ with $m \leq n$ such that $X \in \G^{\leq n} \cap \G^{\geq m}$.

\s Let $\Ac$ be an abelian category. We index $\Ac$-complexes cohomologically, i.e.,
\[
\Xb \colon     \cdots \rt X^n \rt X^{n+1} \rt X^{n+2} \rt \cdots
\]
 let $\D = \D^b(\Ac)$ be its bounded derived category. The standard  bounded $t$-structure on $\D$ is defined as follows:
\[
\D^{\leq 0}  = \{ \Xb \mid \Xb \cong \Yb \ \text{where} \ Y^n = 0 \ \text{for} \ n > 0 \} \ \text{and}
\]
\[
\D^{\geq 0}  = \{ \Xb \mid \Xb \cong \Yb \ \text{where} \ Y^n = 0 \ \text{for} \ n < 0 \}.
\]

\s Let $A$ be a Noetherian ring  of dimension $d$ and let $\D = \D^b(A)$ be its bounded derived category.  Let $\D_i = \D_i^b(A)$ denote the thick subcategory of $\D^b(A)$ consisting of complexes $\Xb$ with $\dim H^n(\Xb) \leq i$ for all $n$. Set $\D_{-1}^b(A) = 0$. Then note that for $i = 0, \ldots, d$ the following defines a bounded $t$-structure on $\D_i$;
\[
\D_i^{\leq 0}  = \{ \Xb  \in \D_i \mid \Xb \cong \Yb \ \text{where} \ Y^n = 0 \ \text{for} \ n > 0 \} \ \text{and}
\]
\[
\D_i^{\geq 0}  = \{ \Xb \in \D_i \mid \Xb \cong \Yb \ \text{where} \ Y^n = 0 \ \text{for} \ n < 0 \}.
\]

\s Consider $\C_i = \C_i(A) = \D_i/\D_{i-1}$. Set
\[
\C_i^{\leq 0}  = \{ \Xb  \in \C_i \mid \Xb \cong \Yb \ \text{where} \ Y^n = 0 \ \text{for} \ n > 0 \} \ \text{and}
\]
\[
\C_i^{\geq 0}  = \{ \Xb \in \C_i \mid \Xb \cong \Yb \ \text{where} \ Y^n = 0 \ \text{for} \ n < 0 \}.
\]
We show
\begin{theorem}\label{t-i}
  $( \C_i^{\leq 0}, \C_i^{\geq 0} )$ defines a bounded $t$-structure on $\C_i$.
\end{theorem}
\begin{proof}
  Axiom (T2) is clearly satisfied by $( \C_i^{\leq 0}, \C_i^{\geq 0} )$.

  Axiom (T3). Let $\Zb \in \C_i$. We can consider $\Zb$ as an element in $\D_i$. Using the bounded $t$-structure on $\D_i$ there exists a triangle in $\D_i$
  where
  $ \Xb \rt \Zb \rt \Yb[-1] \rt \Xb[1]$ where $\Xb \in \D_i^{\leq 0}$ and $\Yb \in \D_i^{\geq 0}$.
  We consider the image of this exact triangle in $\C_i$. We note the image of $\Xb$ is in $\C_i^{\leq 0}$ and the image of $\Yb$ is in $\C_i^{\geq 0}$. Thus the pair $( \C_i^{\leq 0}, \C_i^{\geq 0} )$ satisfies (T3).

(T1)  We show that $\Hom_{\C_i}(\Xb, \Yb[-1]) = 0$ where $\Xb$ is in $\C_i^{\leq 0}$ and the image of $\Yb$ is in $\C_i^{\geq 0}$.
We may assume that $X_i = 0$ for $i > 0$ and $Y_i = 0$ for $i < 0$.

  A morphism $\phi \colon \Xb \rt \Yb[-1] $ in $\C_i$ can be written as a left fraction
\[
\xymatrix{
\
&\Wb
\ar@{->}[dl]_{u}
\ar@{->}[dr]^{f}
 \\
\Xb
\ar@{->}[rr]_{\phi = f \circ u^{-1}}
&\
&\Yb[-1]
}
\]
where $u, f$ are  maps in $\D_i$
and cone $u$ is in $\D_{i-1}$. It follows that $H^n(\Wb)$ has dimension $\leq i -1$ for $n > 0$. Consider the truncation of $\Wb$
\[
\Wb^\prime \colon \cdots W^{-1} \rt W^0 \rt \image \partial_0 \rt 0
\]
We have an obvious map $\theta \colon \Wb^\prime \rt \Wb$ in $\D_i$ with cone of $w$ in $D_{i-1}$. So we have
$$ \phi = f \circ u^{-1} = f \circ \theta \circ (u \circ \theta)^{-1}. $$
We have the following left fraction where $v = u \circ \theta$ and $g = f\circ \theta$ are maps in $\D_i$,
\[
\xymatrix{
\
&\Wb^\prime
\ar@{->}[dl]_{v}
\ar@{->}[dr]^{g}
 \\
\Xb
\ar@{->}[rr]_{\phi = g \circ v^{-1}}
&\
&\Yb[-1]
}
\]
Let $\xi \colon \Pb \rt \Wb^\prime$ be a projective resolution. As $H^n(\Wb^\prime) = 0$ for $n > 0$ we may assume that $P^n = 0$ for $n > 0$.
So we have
$$ \phi = g \circ v^{-1} = g \circ \xi \circ (v \circ \xi)^{-1}. $$
We have the following left fraction where $w = v \circ \xi$ and $h = g \circ \xi$ are maps in $\D_i$,
\[
\xymatrix{
\
&\Pb
\ar@{->}[dl]_{w}
\ar@{->}[dr]^{h}
 \\
\Xb
\ar@{->}[rr]_{\phi = h \circ w^{-1}}
&\
&\Yb[-1]
}
\]
Let $\K$ be the homotopy category of complexes $\Zb$ of finitely generated $A$-modules such that $H^i(\Zb) = 0$ for $i \gg 0$. Then note that
$$\Hom_{\D}(\Pb , \Yb[-1])  = \Hom_\K(\Pb, \Yb[-1]).$$
But $P^n = 0$ for $n > 0$ and $\Yb[-1]^n = 0$ for $n \leq 0$. It follows that \\ $\Hom_\K(\Pb, \Yb[-1]) = 0$. Thus $\phi = 0$.

Finally it is clear that the $t$-structure defined on $\C_i$ is bounded.
\end{proof}

An immediate corollary is
\begin{corollary}\label{split-idemptents} Idempotents split in $\C_i$.
\end{corollary}
\begin{proof}
  The result follows from \cite[Theorem]{LC} as $\C_i$ has a bounded $t$-structure.
\end{proof}
\section{Identifying the heart of $\C_i(A)$}
In this section we identify the heart  $\Hc_i(A)$ of $\C_i = \C_i(A)$ with respect to the $t$-structure defined in the previous section.
Let $\mmod(A)$ denote the category of all finitely generated $A$-modules. Let $\M_i(A)$ denote the category of all finitely generated $A$-modules of dimension $\leq i$. We note that $\M_i(A)$ is a Serre-subcategory of $\mmod(A)$.

\s \label{cf} We first construct an exact functor $\Psi \colon \M_i/\M_{i-1} \rt \Hc_i$.
We note that $\Hc_i = \C_i^{\leq 0} \cap \C_i^{\geq 0}$. Let $M \in \M_i$ then we consider it as a stalk complex $i(M)$ in $\D_i$. Let $j(M)$ be the image of $i(M)$ in $\C_i$. Then note
$j(M) \in \Hc_i$.  Thus we have an additive functor $\Phi \colon \M_r \rt \Hc_i$. Furthermore if $M \in \M_{i-1}$ then $j(M) = 0$ in $\C_i$.
Let $0 \rt M_1 \rt M_2 \rt M_3 \rt 0$ be an exact sequence in $\M_i$. Then note we have an exact triangle $i(M_1) \rt i(M_2) \rt i(M_3) \rt i(M_1)[1]$ in $\D_i$. It's image in
$\C_i$ is $j(M_1) \rt j(M_2) \rt j(M_3) \rt j(M_1)[1]$. As all the complexes is in $\Hc_i$ it follows that we have an exact sequence
$0 \rt j(M_1) \rt j(M_2) \rt j(M_3) \rt 0$ in $\Hc_i$. Thus $\Phi$ is an exact functor. As $j(M) = 0$ when $M \in \M_{i-1}$ we get that $\Phi$ factors as an exact functor
$\Phi \colon \M_i/\M_{i-1} \rt \Hc_i$.

We prove:
\begin{theorem}
(with hypotheses as in \ref{cf}) The functor $\Psi$ is an equivalence of categories.
\end{theorem}
 \begin{proof}
 We have to show $\Psi$ is fully faithful and dense.

 We first show $\Psi$ is dense. Let $\Wb \in \Hc_i$. Then note $\dim H^n(\Wb) \leq i - 1$ for $n \neq 0$ and $\dim H^0(\Wb) \leq i$.
  Consider the truncation of $\Wb$
\[
\Wb^\prime \colon \cdots W^{-1} \rt W^0 \rt \image \partial_0 \rt 0
\]
We have an obvious map $\theta \colon \Wb^\prime \rt \Wb$ in $\D_i$ with cone of $w$ in $D_{i-1}$. So $\Wb \cong \Wb^\prime$ in $\C_i$.
Let $\Pb \rt \Wb^\prime$ be a projective resolution. As $H^n(\Wb^\prime) = 0$ for $n > 0$ we may assume $P^n = 0$ for $n  > 0$. Note
$\Wb \cong \Wb^\prime \cong \Pb$ in $\C_i$. Consider the obvious map
$\Pb \rt P^0/\image \partial^{-1} = U$.
Then considering $U$ as a stalk complex note $\Pb \rt U$ is an isomorphism in $\C_i$. So $\Wb \cong \Psi(U)$.

Next we show $\Psi$ is a full and faithful functor.\\
First we show that if $\theta \colon M \rt N$ a map in $\M_r$ is an isomorphism in $\M_r/\M_{r-1}$ then $\Phi(\theta) \colon j(M) \rt j(N)$ is an isomorphism in $\Hc_i$.
Using $\theta$ we get exact sequences $0 \rt K \rt M \xrightarrow{\alpha} I \rt 0$ and $0 \rt I \xrightarrow{\beta} N \rt C \rt 0$ where $K, C \in \M_{r-1}$ and $\theta = \beta \circ \alpha$.  It follows that we have a triangle
$$j(K) \rt j(M) \xrightarrow{\Psi(\alpha)} j(I) \rt j(K)[1] \ \text{ in $\C_i$.}$$  As $j(K) = 0$ in $\C_i$ we get $\Psi(\alpha) \colon j(M) \rt j(I)$ is an isomorphism. Similarly we get $ \Psi(\beta) \colon j(I) \rt j(N)$ is an isomorphism. It follows $\Psi(\theta) = \Psi(\beta)\circ \Psi(\alpha)$ is an isomorphism in $\C_i$ and so in $\Hc_i$.

Let $M, N \in \M_i$.
We want to show $$\Psi \colon \Hom_{\M_i/\M_{i-1}}(M, N) \rt \Hom_{\Hc_i}(j(M), j(N)) \quad \text{is an isomorphism.} $$
Let $\phi \colon M \rt N$ be a map in $\M_i/\M_{i-1}$. Then $\phi$ can be expressed as a left fraction
\[
\xymatrix{
\
&L
\ar@{->}[dl]_{w}
\ar@{->}[dr]^{h}
 \\
M
\ar@{->}[rr]_{\phi = h \circ w^{-1}}
&\
& N
}
\]
where $w$ is an isomorphism in $\M_i/\M_{i-1}$. Suppose $\Psi(\phi) = 0$. Then $\Psi(h) \circ \Psi(w)^{-1} = 0$. So $\Psi(h) = 0$. By \cite[2.1.26]{N},
$\Psi(h) \colon j(L) \rt j(N)$ factors through $\Xb \in \D_{i-1}$. Taking cohomology we get that $h \colon L \rt N$ factors through $H^0(\Xb) \in \M_{i-1}$. So $h = 0$ in $\M_i/\M_{i-1}$. Thus $\phi = 0$. So it follows that $\Psi$ is faithful.

Next we show $\Psi$ is full.  Let $\psi \colon j(M) \rt j(N)$ be a morphism in $\Hc_i$. Then $\psi$ can be represented by a left fraction
\[
\xymatrix{
\
&L
\ar@{->}[dl]_{w}
\ar@{->}[dr]^{h}
 \\
j(M)
\ar@{->}[rr]_{\psi = h \circ w^{-1}}
&\
& j(N)
}
\]
We note that as discussed before we may assume that $H^i(L) = 0$ for $i > 0$. We may replace $L$ by its projective resolution $\Pb$ and further we may assume $\Pb^n = 0$ for $n > 0$.
We note $w \colon P \rt j(M)$ factors through ${\Qb}_{H^0(\Pb)}$ where ${\Qb}_{H^0(\Pb)}$ is a projective resolution of $H^0(\Pb)$. A similar result holds for $h$. We further note that as $H^n(\Pb)$ has dimension $< i$ for $n < 0$ we get that the natural map $\Pb \rt {\Qb}_{H^0(P)}$ is an isomorphism in $\C_i$. Thus we get a left fraction description of $\psi$ as
\[
\xymatrix{
\
&{\Qb}_{H^0(\Pb)}
\ar@{->}[dl]_{u}
\ar@{->}[dr]^{v}
 \\
j(M)
\ar@{->}[rr]_{\psi = v \circ u^{-1}}
&\
& j(N)
}
\]
Taking cohomology we get the required result.
 \end{proof}

\section{The Krull-Remak-Schmidt property of $\C_i(A)$}
 In this section we prove Theorem \ref{krs}. By \ref{split-idemptents} idempotents split in $\C_i$. Next we show that any $\Xb \in \C_i$ is a direct sum of indecomposables in $\C_i$.
 \s Let $\Xb \in \D_i$. We set
 \[
 \Supp_i \Xb = \{ P \in \Spec(A) \mid \dim A/P = i \ \text{and} \ {\Xb}_P \neq 0 \ \text{in} \ D^b(A_P) \}.
 \]
 It is clear that $\Supp_i \Xb$ is a finite set and $\Supp_i \Xb = \emptyset$ if and $\Xb \in \D_{i-1}$. Set
 \[
 h(\Xb) = \sum_{P \in \Supp_i \Xb}\left( \sum_{n \in \Z}\ell_{A_P}(H^n(\Xb)_P)   \right).
 \]
 We note that $h(\Xb)$ is a finite non-negative integer and $h(\Xb) = 0$ if and $\Xb \in \D_{i-1}$.

 \s For $\Xb \in \D_i$ let $i(\Xb)$ denote its image in $\C_i$. We show that $h$ descends into a well-define function on $\C_i$.
 \begin{lemma}\label{direct}
   Let $\Xb, \Yb \in \D_i$ and let $\Zb \in \D_{i-1}$. Then the following holds
   \begin{enumerate}[\rm (1)]
     \item If $\Xb \rt \Yb \rt \Z_i \rt \Xb[1]$ is a triangle in $\D_i$ then $h(\Xb) = h(\Yb)$.
     \item If $i(\Xb) \cong i(\Yb)$ in $\C_i$ then $\Supp_i(\Xb) = \Supp_i(\Yb)$ and  $h(\Xb) = h(\Yb)$.
     \item  $\Supp_i(-)$ is well defined in $\C_i$. Furthermore $h(-)$ defines a well-defined function on $\C_i$ and $h(\Wb) = 0$ if and only if $\Wb =0$ in $\C_i$.
     \item $h(\Fb \oplus \Wb) = h(\Fb) + h(\Wb)$ in $\C_i$.
     \item Any $\Wb \in \C_i$ is a finite direct sum of indecomposables in $\C_i$.
   \end{enumerate}
 \end{lemma}
 \begin{proof}
   (1) Let $P$ be a prime such that $\dim A/P = i$. We note that ${\Zb}_P = 0$ in $\D^b(A_P)$. It follows that $\Supp_i(\Xb) = \Supp_i(\Yb)$ and $h(\Xb) = h(\Yb)$.

   (2) Let $\phi \colon i(\Xb) \rt i(\Yb)$ be an isomorphism in $\C_i$. We have a left fraction
 \[
\xymatrix{
\
&\Wb
\ar@{->}[dl]_{w}
\ar@{->}[dr]^{h}
 \\
\Xb
\ar@{->}[rr]_{\phi = h \circ w^{-1}}
&\
& \Yb
}
\]
We have $cone(w) \in \D_{i-1}$. Also as $\phi $ is an isomorphism it follows that $cone(h) \in \D_{i-1}$. By (1) we have
 $\Supp_i(\Xb) =  \Supp_i(\Wb) =\Supp_i(\Yb)$. Furthermore
 $h(\Xb) = h(\Wb) = h(\Yb)$.

(3) If $\Wb \in \C_i$, define
$\Supp_i(\Wb) = \Supp_i(\Xb)$ and
$h(\Wb) = h(\Xb)$ where $\Xb \in  \D_i$ such that $i(\Xb) = \Wb$. By (2) these terms  are well defined. Furthermore $\Supp_i(\Wb)$ is empty if and only if $h(\Wb)$ is zero if and only if $\Wb = 0$ in $\C_i$.

(4) Choose $\Xb, \Yb \in \D_i$ such that $i(\Xb) = \Fb, i(\Yb)= \Wb$. We note that $i(\Xb \oplus \Yb) = \Fb\oplus \Wb$.
We have
\[
h(\Fb \oplus \Wb) = h(\Xb \oplus \Yb) = h(\Xb) + h(\Yb) = h(\Fb) + h(\Wb).
\]

(5) Suppose it is not possible to decompose some $\Wb$ into a finite direct sum of indecomposables. Choose $\Wb$ with this property such that $h(\Wb)$ is smallest.
Notice $\Wb$ is not indecomposable. Say $\Wb = \Xb \oplus \Yb$ with $\Xb, \Yb$ non-zero in $\C_i$. Then $h(\Wb) = h(\Xb) + h(\Yb)$ with $h(\Xb) \neq 0$ and $h(\Yb) \neq 0$. It follows that
$h(\Xb) < h(\Wb)$ and $h(\Yb) < h(\Wb)$. By our assumption $\Xb, \Yb$ are a finite direct sum of indecomposables. It follows that $\Wb$ is also a finite direct sum of indecomposables, a contradiction. The result follows.
 \end{proof}

 \s \label{S-set} Let $\Xb, \Yb \in \C_i$. Suppose $\Supp_i(\Xb), \Supp_i(\Yb) \subseteq \{ P_1, \ldots, P_m \} $ with \\ $\dim A/P_j = i$ for $j = 1, \ldots, m$.  Let $S = A \setminus \cup_{j = 1}^{m} P_j$. Suppose
 $\phi \colon \Xb \rt \Yb$ be given as a left fraction
  \[
\xymatrix{
\
&\Wb
\ar@{->}[dl]_{w}
\ar@{->}[dr]^{h}
 \\
\Xb
\ar@{->}[rr]_{\phi = h \circ w^{-1}}
&\
& \Yb
}
\]
Then we note that $\Supp_i(\Wb) = \Supp_i(\Xb)$ and $cone(w)$ is in $D_{i-1}$. Thus we get that $S^{-1}w \colon S^{-1}\Wb \rt S^{-1}\Xb$ is an isomorphism.
So we have a well-defined $A$-linear map
\[
\eta_S(\Xb, \Yb) \colon \Hom_{\C_i}(\Xb, \Yb) \rt \Hom_{\D^b(S^{-1}A)}(S^{-1}\Xb, S^{-1}\Yb).
\]
given by $\phi \mapsto S^{-1}h \circ (S^{-1}w)^{-1}$.
We show
\begin{proposition}
\label{S}( with hypotheses as in \ref{S-set}) $\Hom_{\C_i}(\Xb, \Yb))$ is an $S^{-1}A$-module and $\eta_S(\Xb, \Yb)$ is an isomorphism of $S^{-1}A$-modules.
\end{proposition}
We need a few preliminaries to prove this result.
We first show
\begin{lemma}
  \label{d-local} Let $\Xb, \Yb \in \D_i$. Suppose $\Supp_i(\Xb), \Supp_i(\Yb) \subseteq \{ P_1, \ldots, P_m \} $ with $\dim A/P_j = i$ for $j = 1, \ldots, m$.  Let $S = A \setminus \cup_{j = 1}^{m} P_j$. Then the natural map
  $$ S^{-1}\Hom_{\D^b(A)}(\Xb, \Yb) \rt \Hom_{\D^b(S^{-1}A)}(S^{-1}\Xb, S^{-1}\Yb), $$
  is an isomorphism.
\end{lemma}
\begin{proof}
Let $\K$ denote the homotopy category of complexes $\Zb$ with $\Zb^n = 0$ for $n \gg 0$, $\Zb^n$ finitely generated $A$-module for all $n$ and $H^n(\Zb) = 0$ for all $n \ll 0$.
Let $\Pb$ be a projective resolution of $\Xb$. Then
$$\Hom_{\D^b(A)}(\Xb, \Yb) = \Hom_\K(\Pb, \Yb) = H^0(\Hom_A(\Pb, \Yb)).$$
Note we are assuming that $\Yb$ is a bounded complex. So note that $\Hom_A(\Pb, \Yb)^n $ is finitely generated $A$-module for all $n \in \Z$. So we have
\begin{align*}
  S^{-1}\Hom_{\D^b(A)}(\Xb, \Yb) &=  S^{-1}(H^0(\Hom_A(\Pb, \Yb)) \\
  &\cong H^0( S^{-1}\Hom_A(\Pb, \Yb)) \\
  & \cong H^0(\Hom_{S^{-1}A}(S^{-1}\Pb, S^{-1}\Yb)\\
  &= \Hom_{\D^b(S^{-1}A)}(S^{-1}\Xb, S^{-1}\Yb).
\end{align*}
The last equality holds as $S^{-1}\Pb$ is a projective resolution of $S^{-1}\Xb$.
\end{proof}
Next we show
\begin{lemma}
  \label{s-local} Let $\Xb \in \D_i$. Suppose $\Supp_i(\Xb) \subseteq \{ P_1, \ldots, P_m \} $ with $\dim A/P_j = i$ for $j = 1, \ldots, m$.  Let $S = A \setminus \cup_{j = 1}^{m} P_j$. Let
  $\mu_s \colon \Xb \rt \Xb$ be mutiplication by $s$. Then $cone(\mu_s) \in \D_{i-1}$.
\end{lemma}
\begin{proof}
  Let $Q$ be a prime ideal such that $\dim A/Q = i$. If $Q \neq P_i$ for all $i$ then ${\Xb}_Q = 0$. So $cone(\mu_s)_Q = 0$.
  If $Q = P_i$ for some $i$ then as $s$ is a unit in $A_Q$ we have that the induced map $H^n(\mu_s) \colon H^n({\Xb}_Q) \rt H^n({\Xb}_Q)$ is an isomorphism for all $n \in \Z$. So
  $H^n(cone(\mu_s)_Q) = 0$ for all $n \in \Z$. It follows that $cone(\mu_s)_Q = 0$. Thus $cone(\mu_s) \in \D_{i-1}$.
\end{proof}

We now give
\begin{proof}[Proof of Proposition \ref{S}]
We first show that $\Hom_{\C_i}(\Xb, \Yb)$ is a $S^{-1}A$-module. It suffices to show that $\mu_s \colon \Yb \rt \Yb$ is invertible for every $s \in S$. This follows from Lemma \ref{s-local}.
As $\eta$ is an $A$-linear map between two $S^{-1}A$-modules it is infact $S^{-1}A$-linear.

The natural map $S^{-1}\Hom_{\D^b(A)}(\Xb, \Yb) \rt \Hom_{\D^b(S^{-1}A)}(S^{-1}\Xb, S^{-1}\Yb)$ factors through $\eta(\Xb, \Yb)$. By Lemma \ref{d-local} it follows that $\eta(\Xb, \Yb)$vis surjective. We show $\eta(\Xb, \Yb)$ is injective. Let $\phi = h \circ w^{-1} \colon \Xb \rt \Yb$ be a map with $\eta(\phi) = 0$. We note that $\eta(\phi) = S^{-1}h \circ (S^{-1}w)^{-1}$.
It follows that $S^{-1}h = 0$. By \ref{d-local} it follows that there exists $s \in S$ such that $\mu_s \circ h = 0$ in $\Hom_{\D}(\Xb, \Yb)$. But $\mu_s$ is invertible in $\C_i$. So $h = 0$ in $\C_i$. Therefore $\phi = 0$. Thus $\eta(\Xb, \Yb)$ is injective. The result follows.
\end{proof}

As a consequence we have
\begin{corollary}\label{Artin}
Let $\Xb \in \C_i$. Then $E =\Hom_{\C_i}(\Xb, \Xb)$ is a (left) Artinian ring. If $\Xb$ is indecomposable then $E$ is  local (possibly non-commutative).
\end{corollary}
\begin{proof}
  Suppose $\Supp_i(\Xb) \subseteq \{ P_1, \ldots, P_m \} $ with $\dim A/P_j = i$ for $j = 1, \ldots, m$.  Let $S = A \setminus \cup_{j = 1}^{m} P_j$.
  We note that  the ideals $S^{-1}P_j$ are precisely the maximal ideals in $S^{-1}A$. It follows that $\Hom_{\D^b(S^{-1}A)}(S^{-1}\Xb, S^{-1}\Xb)$ has finite length as a $S^{-1}A$ module and so left Artinian. By \ref{S} it follows that $E$ is a left Artinian ring.

 \s  As idempotents split in $\C_i$ it follows that if $\Xb$ is indecomposable then $E$ has no idempotents, see \cite[1.1]{LW}.
  As $E$ is an Artinian ring it follows that $E$ is local (use \cite[1.5]{LW}
\end{proof}
As a consequence we have
\begin{theorem}
\label{krs-body} For all $i$, $\C_i(A)$ is a Krull-Remak-Schmidt(KRS) category
\end{theorem}
\begin{proof}
By \ref{direct} any $\Xb$ decomposes as a finite direct sum of indecomposables. By \ref{Artin} endomorphism rings of indecomposables are local. Thus $\C_i$ is a KRS-category.
\end{proof}

\section{Proof of Theorem \ref{decomp}}
In this section we prove Theorem \ref{decomp} and give a proof of Corollary \ref{ar}. We need the following well-known result.
\s
\label{d-prelim}
Let $A$ be a semi-local ring with maximal ideals $\m_1, \ldots, \m_r$. Note $ \D^b_0(A)$ is a Krull-Remak-Schmidt category. For $\Xb \in \D^b_0(A)$ note that $\Supp_0(\Xb) \subseteq \{\m_1, \ldots, \m_r \}$. Furthermore if $\Xb \in \D^b_0(A)$ is indecomposable then $\Supp_0(\Xb) = \{ \m_j \} $ for some $j$. If $\Yb \in \D^b_0(A)$ and if  $\Supp_0(\Xb) \cap \Supp_0(\Yb) = \emptyset$
then $\Hom_{\D^b(A)}(\Xb, \Yb) = 0$. Furthermore we have $\D^b_0(A) \cong \bigoplus_{j = 1}^{s} \D_{\m_j}$ where $\D_{\m_j}$ is the thick subcategory in $\D^b_0(A)$ consisting of complexes $\Xb$ with $\Supp_0(\Xb) \subseteq \{\m_j \}$.

\begin{lemma}\label{indecomp-supp}
Let $\Xb$ be indecomposable in $\C_i$. Then $\Supp_i(\Xb) = \{ P \}$ for some prime $P$ with $\dim A/P = i$.
\end{lemma}
\begin{proof}
Say $\Supp_i(\Xb) = \{ P_1, \ldots, P_r \}$ for some primes $P_j$ with $\dim A/P_j = i$ for all $j$. Let $S = A \setminus \bigcup_{j =1}^{r} P_j$. We note that $S^{-1}A$ is semi-local with maximal ideals $S^{-1}P_j$ where $j =1, \ldots, r$. We note that $S^{-1}\Xb \in D^b_0(S^{-1}A)$ and by \ref{d-local} is indecomposable. By \ref{d-prelim} it follows that $r = 1$.
\end{proof}
Next we show
\begin{lemma}\label{split}
  Let $\Xb, \Yb \in \C_i$ be indecomposable. Let $\Supp_i(\Xb) =\{  P \}$ and $\Supp_i(\Yb) = \{ Q \}$. If $P \neq Q$ then $\Hom_{\C_i}(\Xb, \Yb) = 0$.
\end{lemma}
\begin{proof}
Let $S = A \setminus (P \cup Q)$. Then $S^{-1}A$ is semi-local with maximal ideals $S^{-1}P$ and $S^{-1}Q$. We note that $\Supp_0 S^{-1}\Xb = \{ S^{-1} P \}$ and $\Supp_0 S^{-1}\Yb = \{S^{-1}Q \}$.
By \ref{d-prelim} we get $\Hom_{D^b(S^{-1}A)}(S^{-1}\Xb, S^{-1}\Yb) = 0$. The result follows from \ref{d-local}.
\end{proof}

\s\label{dec-P} Let $P$ be a prime ideal in $A$ with $\dim A/P = i$. Let $\V(P) = $ be the full subcategory of complexes $\Xb$ in $\C_i(A)$ with $\Supp_i(\Xb) \subseteq \{ P \}$. It is elementary to see that
$\V(P)$ is a thick subcategory of $\C_i(A)$. As $\C_i(A)$ is a Krull-Schmidt category, it follows from Lemmas \ref{indecomp-supp} and \ref{split} that
$$\C_i(A) =  \bigoplus_{\stackrel{P}{\dim A/P = i}} \V(P).$$

Next we give
\begin{proof}[Proof of Theorem \ref{decomp}]
By \ref{dec-P} it suffices to prove that $\V(P) \cong \D^b_0(A_P)$. We have the obvious functor
$\Phi \colon \V(P) \rt \D^b_0(A_P)$ given by $\Xb \mapsto { \Xb}_P$. By \ref{S} it follows that $\Phi$ is full and faithful. It suffices to prove $\Phi$ is dense.
Let $\Yb \in \D^b_0(A_P)$. We have to show existence of a complex $\Xb$ in $\V(P)$ with $\Phi(\Xb) = \Yb$. We prove this by inducting on $\ell(H^*(\Yb))$.

If $\ell(H^*(\Yb)) = 1$ then note $\Yb \cong \kappa(P)[c]$ for some $c \in \Z$ (here $\kappa(P) = $ residue field of $A_P$). Let $\Xb = A/P[c]$ a stalk complex concentrated in degree $c$.
Notic $\Phi(\Xb) = \Yb$.

Now let $\ell(H^*(\Yb)) = r$  and the result is known for complexes $\Zb$ with \\  $\ell(H^*(\Zb)) < r$.  Let $\Pb$ be a projective resolution of $\Yb$, we may assume (after shifting) that $H^0(\Pb) \neq 0$ and $\Pb^n = 0$ for $n > 0$. Let $\Qb$ be projective resolution of $\kappa(P)$. We have an obvious surjective  map $H^0(\Pb) \rt \kappa(P)$. This yields a chain map $\Pb \rt \Qb$.
After taking a sufficiently high good truncation $\Zb$ of $\Qb$ we get a triangle in $\D^b_0(A_P)$
\[
\Yb \rt \Zb \xrightarrow{f} \Wb \rt \Yb[1]
\]
where $\ell(H^*(\Wb)) = r -1$. By induction hypotheses there exist $\Ub, \Vb \in \V(P)$ such that there exists isomorphisms $g \colon \Phi(\Ub) \rt \Zb$ and $h \colon \Phi(\Vb) \rt \Wb$.
The map  $h^{-1} \circ f \circ g = \Phi(\theta)$ for some $\theta \colon \Ub \rt \Vb$ (as $\Phi$ is fully-faithful). We have a triangle in $\V(P)$
\[
\Kb \rt \Ub \xrightarrow{\theta} \Vb \rt \Kb[1].
\]
It is elementary to verify that $\Phi(\Kb) \cong \Yb$. The result follows.
\end{proof}
We now give
\begin{proof}
  If $A$ is regular then so is $A_P$ for any prime $P$. By \cite[1.1]{P2},  $D^b_0(A_P)$ has AR-triangles. The result follows.
\end{proof}

\end{document}